\documentclass[12pt,reqno]{amsart}

\usepackage{amssymb}

\newtheorem{theorem}{Theorem}[section]

\newtheorem{corollary}[theorem]{Corollary}

\theoremstyle{definition}
\newtheorem{definition}[theorem]{Definition}

\numberwithin{equation}{section}

\begin{document}

\baselineskip=15pt

\title[Holomorphic $\text{GL}_2({\mathbb C})$-geometry on compact complex 
manifolds]{Holomorphic $\text{GL}_2({\mathbb C})$-geometry on compact complex manifolds }

\author[I. Biswas]{Indranil Biswas}

\address{School of Mathematics, Tata Institute of Fundamental
Research, Homi Bhabha Road, Mumbai 400005, India}

\email{indranil@math.tifr.res.in}

\author[S. Dumitrescu]{Sorin Dumitrescu}

\address{Universit\'e C\^ote d'Azur, CNRS, LJAD, France}

\email{dumitres@unice.fr}



\date{}
\begin{abstract}
We study holomorphic $\text{GL}_2({\mathbb C})$ and $\text{SL}_2({\mathbb C})$ geometries on 
compact complex manifolds.
\end{abstract}

\maketitle

\tableofcontents

\section{Introduction}

A holomorphic $\text{GL}_2({\mathbb C})$ geometric structure on a complex manifold $X$ of 
complex dimension $n$ is a holomorphic point-wise identification between the holomorphic 
tangent space $TX$ and homogeneous polynomials in two variables of degree $(n-1)$. More 
precisely, a $\text{GL}_2({\mathbb C})$ geometric structure on $X$ is a pair $(E,\, 
\varphi)$, where $E$ is a rank two holomorphic vector bundle on $X$ and $\varphi$ is a 
holomorphic vector bundle isomorphism of $TX$ with the $(n-1)$--fold symmetric product 
$S^{n-1}(E)$ (see Definition \ref{def}). If $E$ has trivial determinant (i.e., the 
holomorphic line bundle $\bigwedge^2E$ is trivial), then $(E,\, \varphi)$ is called a 
holomorphic $\text{SL}_2({\mathbb C})$ geometric structure.

The above definitions are the holomorphic analogues of the concepts of 
$\text{GL}_2({\mathbb R})$ and $\text{SL}_2({\mathbb R})$ geometries in the real smooth category (for the study 
of those geometries in the real smooth category we refer
the reader to \cite{DG, FK, Kr} and references therein).

This article deals with the classification of compact complex manifolds admitting holomorphic 
$\text{GL}_2({\mathbb C})$ and $\text{SL}_2({\mathbb C})$ geometries.

Holomorphic $\text{GL}_2({\mathbb C})$ and $\text{SL}_2({\mathbb C})$ geometries on $X$ are examples 
of holomorphic $G$--structures (see \cite{Kob} and Section \ref{contexte}). They correspond to the reduction of 
the structural group of the $\text{GL}_n({\mathbb C})$--frame bundle of $X$ to $\text{GL}_2({\mathbb C})$ and 
$\text{SL}_2({\mathbb C})$ respectively.

When the dimension $n$ of $X$ is odd, then a $\text{GL}_2({\mathbb C})$ geometry produces
a {\it holomorphic conformal structure} on $X$ (see, for example, \cite{DG}, Proposition 3.2 or Section \ref{contexte} here). Recall 
that a holomorphic conformal structure is defined by a holomorphic line bundle $L$ over $X$ and a holomorphic 
section of $S^2(T^*X) \otimes L$, which is a $L$-valued fiberwise nondegenerate holomorphic quadratic form on $TX$. 
Moreover, if $n\,=\, 3$, then a $\text{GL}_2({\mathbb C})$-geometry on $X$ is exactly a 
holomorphic conformal structure. The standard flat example is the smooth quadric $Q_3$ in ${\mathbb C}{\mathbb
P}^4$ defined by the equation $Z_0^2+Z_1^2+Z^2_2 + Z^2_3+ Z_4^2\,=\,0$.

A $\text{SL}_2({\mathbb C})$-geometry on $X$ defines a {\it holomorphic Riemannian 
metric} (i.e. a holomorphic section of $S^2(T^*X)$ which is point-wise nondegenerate \cite{Du})
when $n$ is odd. When $n\,=\,3$, a $\text{SL}_2({\mathbb C})$-geometry is the same data
as a holomorphic Riemannian metric.

When the dimension $n$ of $X$ is even, 
a holomorphic $\text{SL}_2({\mathbb C})$--geometry on 
$X$ produces a nondegenerate holomorphic two form $\omega$ on $X$. Moreover, if $X$ is a compact 
K\"ahler manifold, then $\omega$ is automatically closed, and hence $\omega$ is a
{\it holomorphic symplectic form}, and therefore $X$ is a hyper-K\"ahler manifold (see \cite{Be}).

When $n$ is even, a $\text{GL}_2({\mathbb C})$--geometry on $X$ produces a
twisted nondegenerate two form $\omega$ on $X$, which means that there is a holomorphic line bundle $L$ over 
$X$ such that $\omega$ is a fiberwise nondegenerate holomorphic global section of $\Omega^2_X \otimes L$.
We use the terminology of \cite{Is1, Is2} and call $\omega$ a {\it twisted holomorphic symplectic form}.

Let us describe the results of this article. In Section \ref{Kahler and C} we prove Theorem 
\ref{kahler even} asserting that {\it a compact K\"ahler manifold of even complex dimension 
$n \,\geq\, 4$ admitting a holomorphic $\text{GL}_2({\mathbb C})$--geometry is covered by a 
compact torus}. A key ingredient of the proof is a result of Istrati \cite{Is1, Is2} proving 
that K\"ahler manifolds bearing a twisted holomorphic symplectic form have vanishing first 
Chern class. By Yau's proof of Calabi's conjecture \cite{Ya} these manifolds admit a Ricci 
flat metric and the canonical line bundle $K_X$ is trivial (up to a finite cover) \cite{Bo1, 
Be}. Hence the $\text{GL}_2({\mathbb C})$--geometry is induced by an underlying 
$\text{SL}_2({\mathbb C})$--geometry. The proof of Theorem \ref{kahler even} involves 
showing that {\it any compact K\"ahler manifold of complex dimension $n \geq 3$ bearing a 
${\rm SL}_2({\mathbb C})$--geometry is covered by a compact complex torus}. Recall that 
compact K\"ahler manifolds of odd complex dimension bearing a $\text{SL}_2({\mathbb 
C})$--geometry also admit a holomorphic Riemannian metric and hence have  the associated 
holomorphic (Levi-Civita) affine connections; such manifolds are known to have vanishing 
Chern classes, \cite{At}, and hence they admit a  covering by some compact complex torus 
\cite{IKO}.

Theorem \ref{Fujiki} deals with manifolds $X$ in Fujiki class $\mathcal C$ (i.e. 
holomorphic images of compact K\"ahler manifolds \cite{Fu}) bearing a $\text{GL}_2({\mathbb 
C})$--geometry. Under the technical assumption that there exists a cohomology class $\lbrack 
\alpha \rbrack\,\in\, H^{1,1}(X,\, \mathbb R)$ which is numerically effective (nef) and has 
positive self-intersection (meaning $\int_X \alpha^{2m} \,>\, 0$, where 
$2m\,=\,\dim_{\mathbb C} X$), we prove that {\it there exists a non-empty Zariski open 
subset $\Omega\,\subset\, X$ admitting a flat K\"ahler metric.} We recall that simply 
connected non-K\"ahler manifolds in Fujiki class $\mathcal C$ admitting a holomorphic 
symplectic form were constructed in \cite[Example 21.7]{Huy} (see also  \cite{Bo2,Gu1,Gu2} for  other  constructions  of  simply connected   non-K\"ahler holomorphic  symplectic manifolds).

In Section \ref{KE and Fano} we obtain the classification of compact K\"ahler-Einstein 
manifolds (Theorem \ref{KE}) and of Fano manifolds (Theorem \ref{Fano}) bearing a 
$\text{GL}_2({\mathbb C})$--geometry. Theorem \ref{KE} states that {\it any compact 
K\"ahler--Einstein manifold, of complex dimension at least three, bearing a ${\rm 
GL}_2({\mathbb C})$--geometry is
\begin{itemize}
\item either covered by a torus,

\item or  biholomorphic to the three dimensional  quadric $Q_3$,

\item or covered by the three dimensional Lie ball $D_3$ (the noncompact dual of $Q_3$, as
Hermitian symmetric space).
\end{itemize}
In all these three situations the ${\rm GL}_2({\mathbb C})$--geometry is the (flat) standard one.}

Theorem \ref{Fano} asserts that {\it a Fano manifold bearing a holomorphic ${\rm GL}_2({\mathbb C})$--geometry 
is isomorphic to the quadric $Q_3$ endowed with its standard ${\rm GL}_2({\mathbb C})$--structure.}

Theorem \ref{KE} and Theorem \ref{Fano} belong to the same circle of ideas as the following known results. 
Kobayashi and Ochiai proved in \cite{KO3} that compact K\"ahler--Einstein manifolds bearing a holomorphic 
conformal structure are the standard ones: quotients of tori, the smooth $n$-dimensional quadric $Q_n$ and the 
quotients of the non compact dual $D_n$ of $Q_n$. Moreover the same authors proved in \cite{KO2} that 
holomorphic $G$--structures, modeled on an irreducible Hermitian symmetric space of rank $\geq\, 2$ (in particular, 
a holomorphic conformal structure), on compact K\"ahler-Einstein manifolds are always flat. Let us also mention the main 
result in \cite{HM} which says that holomorphic irreducible reductive $G$--structures on uniruled projective 
manifolds are always flat. Consequently, a uniruled projective manifold bearing a holomorphic conformal 
structure is biholomorphic to the quadric $Q_n$ (see also \cite{Ye}). The following generalization was proved in 
\cite{BM}: all holomorphic Cartan geometries (see \cite{Sh}) on manifolds admitting a rational curve are flat.
 
Let us clarify that the classification obtained in Theorem \ref{KE} and Theorem \ref{Fano} does not use results 
coming from \cite{BM,HM, KO2,KO3, Ye}: the methods are specific to the case of $\text{GL}_2({\mathbb 
C})$--geometry and unify the twisted holomorphic symplectic case (even dimension) and the holomorphic conformal 
case (odd dimension).

The last section discusses  some related open problems. Let us emphasize one of those questions dealing with 
$\text{SL}_2({\mathbb C})$--geometries on non-K\"ahler manifolds.

We recall that Ghys constructed in \cite{Gh} exotic deformations of quotients of $\text{SL}_2({\mathbb C})$ by 
normal lattices. Those Ghys manifolds are non-K\"ahler and non-parallelizable, but they admit a non-flat 
holomorphic $\text{SL}_2({\mathbb C})$--geometric structure (i.e. a holomorphic Riemannian metric). 
Nevertheless, all Ghys holomorphic Riemannian manifolds are locally homogeneous. Moreover, it was proved in 
\cite{Du} that holomorphic Riemannian metrics on compact complex threefolds are always locally homogeneous.

For higher complex odd dimensions we conjecture that $\text{SL}_2({\mathbb C})$--geometries on compact complex 
manifolds are always locally homogeneous. Some evidence for it was recently provided by the main 
result in \cite{BD} showing that compact complex simply connected manifolds do not admit any holomorphic Riemannian 
metric (in particular, in odd dimension, they do not admit a holomorphic $\text{SL}_2({\mathbb C})$--geometry).

\section{Holomorphic $\text{GL}_2({\mathbb C})$ and $\text{SL}_2({\mathbb C})$ 
geometries}\label{contexte}

In this section we introduce the framework of $\text{GL}_2({\mathbb C})$ and 
$\text{SL}_2({\mathbb C})$ geometries on complex manifolds and describe the geometry of the 
quadric (model of the flat holomorphic conformal geometry) 
and that of its noncompact dual. We focus on the complex dimension three, the only case admitting a 
$\text{GL}_2({\mathbb C})$-geometry.

The holomorphic tangent bundle of a complex manifold $Z$ will be denoted by $TZ$. The frame bundle
for $TZ$ will be denoted by $R(Z)$. We recall that $R(Z)$ consists of all linear isomorphisms from
${\mathbb C}^d$ to the fibers of $TZ$, where $d\,=\, \dim_{\mathbb C} Z$. The $R(Z)$ is a holomorphic
principal $\text{GL}(d,{\mathbb C})$--bundle over $Z$.

\begin{definition}\label{def}
A holomorphic $\text{GL}_2({\mathbb C})$-{\it structure} (or 
$\text{GL}_2({\mathbb C})$-{\it geometry}) on a complex manifold $X$ of complex dimension $n\,
\geq\, 2$ is a holomorphic bundle isomorphism $TX \,\simeq\, S^{n-1} (E)$, where
$S^{n-1}(E)$ is the $(n-1)$-th symmetric power of a rank two 
holomorphic vector bundle $E$ over $X$. If $E$ has trivial determinant, meaning the line
bundle $\bigwedge^2 E$ is 
holomorphically trivial, it is called a holomorphic $\text{SL}_2({\mathbb C})$-{\it 
structure} (or $\text{SL}_2({\mathbb C})$-{\it geometry}) on $X$.
\end{definition}

Holomorphic $\text{GL}_2({\mathbb C})$-structures and $\text{SL}_2({\mathbb C})$-structures are 
particular cases of holomorphic irreducible reductive $G$-structures \cite{Kob, HM}. They 
correspond to the holomorphic reduction of the structure group of the frame bundle $R(X)$ of $X$ from 
$\text{GL}_n({\mathbb C})$ to $\text{GL}_2({\mathbb C})$ and $\text{SL}_2({\mathbb C})$ 
respectively. It should be clarified that for a $\text{GL}_2({\mathbb C})$-structure, the 
corresponding group homomorphism $\text{GL}_2({\mathbb C}) \,\longrightarrow\, 
\text{GL}_n({\mathbb C})$ is given by the $(n-1)$-th symmetric product of the standard
representation of $\text{GL}_2({\mathbb C})$. This 
$n$-dimensional irreducible linear representation of $\text{GL}_2({\mathbb C})$ is also given by the 
induced action on the homogeneous polynomials of degree $(n-1)$ in two variables. For an
$\text{SL}_2({\mathbb C})$-geometry, the corresponding homomorphism $\text{SL}_2({\mathbb C}) 
\,\longrightarrow\, \text{GL}_n({\mathbb C})$ is the restriction of the above homomorphism to
$\text{SL}_2({\mathbb C})\, \subset\,\text{GL}_2({\mathbb C})$.

The standard symplectic form on ${\mathbb C}^2$ produces a nondegenerate quadratic form (respectively, 
nondegenerate alternating form) on the symmetric product $S^{2i}({\mathbb C}^2)$ (respectively, 
$S^{2i-1}({\mathbb C}^2)$) for all $i\, \geq\, 1$. Consequently, the above linear representation 
$\text{SL}_2({\mathbb C})\,\longrightarrow\, \text{GL}_n({\mathbb C})$ preserves a nondegenerate complex 
quadratic form on ${\mathbb C}^n$ if $n$ is odd and a nondegenerate two (alternating) form on ${\mathbb C}^n$ if 
$n$ is even (see, for example, Proposition 3.2 and Sections 2 and 3 in \cite{DG} or \cite{Kr}). The above 
linear representation $\text{GL}_2({\mathbb C})\,\longrightarrow\, \text{GL}_n({\mathbb C})$ preserves the line 
in ${(\mathbb C}^n)^*\otimes {(\mathbb C}^n)^*$
spanned by the above tensor; the action of $\text{GL}_2({\mathbb C})$ on this line is nontrivial.
Therefore, the $\text{GL}_2({\mathbb C})$-geometry (respectively, 
$\text{SL}_2({\mathbb C})$-geometry) on $X$ induces a {\it holomorphic conformal structure} (respectively, {\it 
holomorphic Riemannian metric}) if the complex dimension of $X$ is odd, and it
induces a {\it twisted holomorphic 
symplectic structure} (respectively, a {\it holomorphic nondegenerate two form})
when the complex dimension of $X$ is even.
 
A holomorphic $\text{SL}_2({\mathbb C})$-structure on a complex surface $X$ is a holomorphic 
trivialization of the canonical bundle $K_X\,=\, \bigwedge^2 (TX)^*$.
 
The simplest nontrivial examples of $\text{GL}_2({\mathbb C})$ and $\text{SL}_2({\mathbb C})$
structures are provided by 
the complex threefolds. In this case a $\text{GL}_2({\mathbb C})$-structure is a holomorphic conformal structure 
on $X$, and an $\text{SL}_2({\mathbb C})$-structure is a holomorphic Riemannian metric on $X$. Indeed, this is 
deduced from the fact that the $\text{SL}_2({\mathbb C})$-representation on the three-dimensional vector space 
of homogeneous quadratic polynomials in two variables
\begin{equation}\label{hp}
\{aX^2+bXY +c Y^2\, \mid\, a,\, b,\, c\, \in\, \mathbb C\}
\end{equation}
preserves the discriminant 
$\Delta\,=\,b^2-4ac$. Consequently, a holomorphic isomorphism between $\text{PSL}_2({\mathbb C})$ and the 
complex orthogonal group $\text{SO}(3, \mathbb C)$ is obtained. In other words, the discriminant, being 
nondegenerate, induces a holomorphic Riemannian metric on the threefold $X$. Notice that a holomorphic Riemannian metric  coincides with a reduction
of the structural group of the frame bundle $R(X)$   to the orthogonal group $\text{O}(3, \mathbb C)$. On a double unramified cover of  a holomorphic Riemannian threefold  there is a reduction of the structural group of the frame bundle 
to the connected component of the identity $\text{SO}(3, \mathbb C)$. Hence in  complex dimension three a  $\text{SL}_2({\mathbb C})$-geometry is the same data as a holomorphic Riemannian metric, while  a reduction of the frame bundle to
$\text{PSL}_2({\mathbb C})$ is the same data as a holomorphic Riemannian metric with an orientation.

Moreover, the $\text{GL}_2({\mathbb C})$-representation on the vector space in \eqref{hp} 
preserves the line generated by the discriminant $\Delta\,=\,b^2-4ac$. This gives an isomorphism 
between $\text{GL}_2({\mathbb C})/({\mathbb Z}/2{\mathbb Z})$ and the conformal group $\text{CO}(3, \mathbb C) 
\,=\,(\text{O}(3, \mathbb C) \times {\mathbb C}^* )/({\mathbb Z}/2{\mathbb Z}) \,=\, \text{SO}(3, \mathbb C) \times {\mathbb C}^*$. Therefore, 
the $\text{GL}_2({\mathbb C})$--structure coincides with a holomorphic reduction of the structure 
group of the frame bundle $R(X)$ to $\text{CO}(3, \mathbb C)$. This holomorphic reduction of the 
structure group defines a holomorphic conformal structure on $X$. Notice that $\text{CO}(3, \mathbb C)$ being connected, the two different orientations
of a three dimensional  holomorphic Riemannian manifold are conformally equivalent.
 
Recall that flat conformal structures in complex dimension $n \geq 3$ are locally modeled on the 
quadric
$$
Q_n\, :=\, \{[Z_0 : Z_1 : \cdots : Z_{n+1}] \, \mid\, Z_0^2+Z_1^2+ \ldots +Z_{n+1}^2\,=\,0\}
\, \subset\, {\mathbb C}{\mathbb P}^{n+1}\, .
$$
The holomorphic automorphism group of $Q_n$ is $\text{PSO}(n+2, \mathbb C)$.
 
Let us mention that $Q_n$ is identified with the real Grassmannian of oriented $2$--planes in
${\mathbb R}^{n+2}$. From this it
follows that $Q_n\,=\, {\rm SO}(n+2, \mathbb R)/({\rm SO}(2, \mathbb R)\times{\rm SO}(n, \mathbb R))$
(see, for instance, Section 1 in \cite{JR2}). The action of ${\rm SO}(n+2, \mathbb R)$ on $Q_n$ is
via holomorphic automorphisms. To see this, first note that the 
real tangent space at a point of $Q_n\,=\, {\rm SO}(n+2, \mathbb R)/({\rm SO}(2, \mathbb R)\times
{\rm SO}(n, \mathbb R))$ is identified 
with the corresponding quotient of the real Lie algebras ${\rm so}(n+2, \mathbb R)/({\rm so}(2, \mathbb R)
\oplus{\rm so}(n, \mathbb R))$, and the stabilizer ${\rm SO}(2, \mathbb R) \times {\rm SO}(n, \mathbb R)$
acts on this quotient vector space through the adjoint 
representation.
The complex structure of the tangent space (induced by the complex structure of $Q_n$) is given 
by the operator $J$ which satisfies the condition that $\{\exp(tJ)\}_{t\in\mathbb R}$
is the adjoint action of the factor ${\rm SO}(2, \mathbb R)$. Since ${\rm SO}(2, \mathbb R)$
lies in the center of the stabilizer, the almost complex structure $J$ is preserved by the action of
the stabilizer. Consequently, the complex structure is ${\rm SO}(n+2, \mathbb R)$--invariant.
 
The quadric $Q_n$ is a Fano manifold, and it is an irreducible Hermitian symmetric space of type \textbf{BD I}
\cite[p.~312]{Bes}. For more about its geometry and that of its noncompact dual 
$D_n$ (as a Hermitian symmetric space) the reader is referred to \cite[Section 1]{JR2}.
 
Theorem \ref{Fano} implies that among the quadrics $Q_n$, $n \,\geq\, 3$,
only $Q_3$ admits a holomorphic $\text{GL}_2({\mathbb C})$--structure. Moreover, Theorem 
\ref{KE} states that the only compact non-flat K\"ahler-Einstein manifolds bearing a holomorphic 
$\text{GL}_2({\mathbb C})$-structure are $Q_3$ and those covered by its noncompact dual $D_3$.

Recall that a general result of Borel on Hermitian irreducible symmetric spaces shows 
that the non compact dual is always realized as an open subset of its compact dual.
 
We will give below a geometric description of the noncompact dual $D_3$ of $Q_3$ as an open 
subset in $Q_3$ which seems to be less known.
 
Consider the complex quadric form $q_{3,2}\,:=\,Z_0^2 +Z_1^2+Z_2^2-Z_3^2-Z_4^2$ of five variables,
and let $$Q\, \, \subset\, {\mathbb C}{\mathbb P}^4$$
be the quadric $Q$ defined 
by the equation $q_{3,2}\,=\,0$. Then $Q$ is biholomorphic to $Q_3$.
Let ${\rm O}(3,2)\,\subset\, \text{GL}(5,{\mathbb R})$ be the real orthogonal group of 
$q_{3,2}$, and denote by ${\rm SO}_0(3,2)$ the connected component
of ${\rm O}(3,2)$ containing the identity element. The quadric $Q$ admits a 
natural holomorphic action of the real Lie group ${\rm SO}_0(3,2)$, which is not transitive,
in contrast to the action of ${\rm SO}(5, \mathbb R)$ on $Q_3$. The orbits of the
$SO_0(3,2)$--action on $Q$ 
coincide with the connected components of the complement $Q  \setminus S$, where $S$ is
the real hypersurface in ${\mathbb 
C}{\mathbb P}^4$  defined by the equation
$$\mid Z_0\mid^2 + \mid Z_1\mid^2 + \mid Z_2\mid^2 - \mid Z_3\mid^2 - 
\mid Z_4\mid^2\,=\,0\, .$$
 
Notice that the above real hypersurface $S$ contains all real points of $Q$. In fact, it can be 
shown that $S \bigcap Q$ coincides with the set of point $m \,\in\, Q$ such that the complex 
line $(m,\, \overline{m})$ is isotropic (i.e. it also lies in $Q$). Indeed, since 
$q_{3,2}(m)\,=\,0$, then the line generated by $(m, \,\overline{m})$ lies in $Q$ if and only if 
$m$ and $\overline{m}$ are perpendicular with respect to the bilinear symmetric form associated 
to $q_{3,2}$, or equivalently $m \,\in\, S$.

For any point $m \,\in\, Q \setminus S$, the form $q_{3,2}$ is nondegenerate on the line $(m,\, 
\overline{m})$. To see this first notice that the complex line generated by $(m,\, 
\overline{m})$, being real, may be considered as a plane in the real projective space ${\mathbb 
R}{\mathbb P}^4$. The restriction of the (real) quadratic form $q_{3,2}$ to this real plane 
$(m,\, \overline{m})$ vanishes at the points $m$ and $\overline{m}$ which are distinct (because 
all real points of $Q$ lie in $S$). It follows that the quadratic form cannot have signature 
$(0,\,1)$ or $(1,\,0)$ when restricted to the real plane $(m,\, \overline{m})$. Consequently, 
the signature of the restriction of $q_{3,2}$ to this plane is either $(2,\,0)$ or $(1,\,1)$ or 
$(0,\,2)$. Each of these three signature types corresponds to an ${\rm SO}_0(3,2)$ orbit in $Q$.
 
Take the point $m_0\,=\, [0:0:0:1:i] \,\in\, Q$. The noncompact dual $D_3$ of $Q_3$ is the 
${\rm SO_0}(3,2)$--orbit of $m_0$ in $Q$. It is an open subset of $Q$ biholomorphic to a 
bounded domain in ${\mathbb C}^3$; it is the three dimension Lie ball (the bounded domain 
$IV_3$ in Cartan's classification).

The signature of $q_{3,2}$ on the above line $(m_0,\,\overline{m}_0)$ is $(0,\,2)$. The 
signature of $q_{3,2}$ on the orthogonal part of $(m_0,\, \overline{m}_0)$, canonically 
isomorphic to $T_{m_0}Q$, is $(3,\,0)$. Then the ${\rm SO}_0(3,2)$--orbit of $m_0$ in $Q$ inherits 
an ${\rm SO}_0(3,2)$--invariant Riemannian metric. The stabilizer of $m_0$ is ${\rm SO}(2, 
\mathbb R) \times{\rm SO}(3, \mathbb R)$. Here ${\rm SO}(2, \mathbb R)$ acts on ${\mathbb 
C}^3$ through the one parameter group $\text{exp}(t J)$, with $J$ being the complex 
structure, while the ${\rm SO}(3, \mathbb R)$ action on ${\mathbb C}^3$ is given by the 
complexification of the canonical action of ${\rm SO}(3, \mathbb R)$ on ${\mathbb R}^3$. 
Hence we conclude that the action of ${\rm U}(2)\,= \,{\rm SO}(2,{\mathbb R})\times{\rm SO}(3, 
{\mathbb R})$ is the unique irreducible action on the symmetric product $S^2({\mathbb 
C}^2)\,=\, {\mathbb C}^3$. This action coincides with the holonomy representation of this 
Hermitian symmetric space $D_3$; as mentioned before, $D_3$ is the noncompact dual of $Q_3$. 
The holonomy representation for $Q_3$ is the same.

Recall that the automorphism group of the noncompact dual $D_3$ is ${\rm PSO}_0(3,2)$; it is 
the subgroup of the automorphism group of $Q$ that preserves $D_3$ (which lies in $Q$ as 
the ${\rm SO}_0(3,2)$--orbit of $m_0 \,\in\, Q$). From this it follows that any quotient of $D_3$ by a 
lattice in ${\rm PSO}_0(3,2)$ admits a flat holomorphic conformal structure induced by that 
of the quadric $Q$. 

Notice that compact projective threefolds admitting holomorphic conformal structures 
($\text{GL}_2({\mathbb C})$-structures) were classified in \cite{JR1}. There are in fact 
only the standard examples: finite quotients of three dimensional abelian varieties, the 
smooth quadric $Q_3$ and quotients of its non compact dual $D_3$. In \cite{JR2}, the same 
authors classified also the higher dimensional compact projective manifolds admitting a flat 
holomorphic conformal structure; they showed that the only examples are the standard ones.

\section{$\text{GL}_2({\mathbb C})$-structures on K\"ahler and Fujiki class $\mathcal C$ 
manifolds}\label{Kahler and C}

Every compact complex surface of course admits a holomorphic $\text{GL}_2({\mathbb 
C})$-structure by taking $E$ in Definition \ref{def} to be the
tangent bundle itself. The situation is much more stringent in higher dimensions. The following 
result shows that a compact K\"ahler manifold of even dimension $n\,\geq\, 4$ bearing a 
holomorphic $\text{GL}_2({\mathbb C})$--structure has trivial holomorphic tangent bundle (up 
to finite \'etale cover).

\begin{theorem}\label{kahler even}
Let $X$ be a compact K\"ahler manifold of even complex dimension $n\,\geq\, 4$ admitting a 
holomorphic ${\rm GL}_2({\mathbb C})$--structure. Then $X$ admits a finite unramified 
covering by a compact complex torus.
\end{theorem}

\begin{proof}
Let $E$ be a holomorphic vector bundle on $X$ and 
$$TX \,\simeq\, S^{n-1} (E)$$
an isomorphism with the symmetric product, 
giving a $\text{GL}_2({\mathbb C})$--structure on $X$. Then
$$
TX \,=\, S^{n-1} (E)\,=\, S^{n-1} (E)^*\otimes (\bigwedge\nolimits^2E)^{\otimes (n-1)}
\,=\, (TX)^*\otimes L\, ,
$$
where $L\,=\, (\bigwedge\nolimits^2E)^{\otimes (n-1)}$. The above isomorphism between
$TX$ and $(TX)^*\otimes L$ produces, when $n$ is even, a holomorphic section
$$
\omega\, \in\, H^0(X,\, \Omega^2_X \otimes L)
$$
which is a fiberwise nondegenerate $2$--form with values in $L$. Writing $n\,=\, 2m$, the
exterior product
$$
\omega^m\,\in \, H^0(X,\, K_X \otimes L^m)
$$
is a nowhere vanishing section, where $K_X\,=\, \Omega^n_X$ is the canonical line bundle on $X$.
Consequently, we have $K_X\,\simeq\,(L^*)^m$, in particular, $c_1(X)\,=\,mc_1(L)$.
Any Hermitian metric on $TX$ induces an associated Hermitian metric on $L^m$, and hence
produces an Hermitian metric on $L$.

We now use a result of Istrati, \cite[p.~747, Theorem 2.5]{Is1}, which says that
$c_1(X)\,=\,c_1(L)\,=\,0$. Hence, by 
Yau's proof of Calabi's conjecture, \cite{Ya}, there exists a Ricci flat  K\"ahler metric $g$ 
on $X$.

Using de Rham decomposition theorem and Berger's classification of the irreducible holonomy groups 
of nonsymmetric Riemannian manifolds (see \cite{Jo}, Section 3.2 and Theorem 3.4.1 in Section 
3.4) we deduce that the universal cover $(\widetilde{X}, \,\widetilde{g}) $ of $(X,\,g)$ is
a Riemannian product
\begin{equation}\label{e1}
(\widetilde{X},\, \widetilde{g})\,=\, ({\mathbb C}^l,\, g_0) \times
(X_1,\, g_1) \times \cdots \times (X_p,\, g_p)\, ,
\end{equation}
where $({\mathbb C}^l, \, g_0)$ is the standard flat complete K\"ahler manifold and $(X_i,\, 
g_i)$ is an irreducible Ricci flat K\"ahler manifold of complex dimension $r_i\, \geq\, 2$, for every $1\, 
\leq\, i\, \leq\, p$. The holonomy of each $(X_i,\, g_i)$ is either $\text{SU}(r_i)$ 
or the symplectic group ${\rm Sp}(\frac{r_i}{2})$, where $r_i\, =\, \dim_{\mathbb C} X_i$
(in the second case $r_i$ is even). Notice, in 
particular, that symmetric irreducible Riemannian manifolds of (real) dimension at least two 
are never Ricci flat. For more details, the reader can refer to \cite[Theorem 1]{Be}
and \cite[p.~124, Proposition 6.2.3]{Jo}.

Using Cheeger--Gromoll theorem it is possible to deduce the  Beauville--Bogomolov decomposition theorem (see 
\cite[Theorem 1]{Be} or \cite{Bo1}) which says that there is a finite \'etale Galois covering
$$
\varphi\, :\, \widehat{X}\, \longrightarrow\, X\, ,
$$
such that
\begin{equation}\label{e2}
(\widehat{X},\, \varphi^*g)\, =\, (T_l,\, g_0) \times
(X_1,\, g_1) \times \cdots \times (X_p,\, g_p)\, ,
\end{equation}
where $(X_i,\, g_i)$ are as in \eqref{e1} and $(T_l,\, g_0)$ is a flat compact complex torus of dimension
$l$. Note that the K\"ahler metric $\varphi^*g$ is Ricci--flat because $g$ is so.

The holomorphic $\text{GL}_2({\mathbb C})$--structure on $X$ produces a holomorphic 
$\text{SL}_2({\mathbb C})$--structure on a finite unramified cover of $\widehat{X}$ in 
\eqref{e2}. More precisely, $T\widehat{X}\,=\, S^{n-1}(\varphi^* E)$ and since the canonical 
bundle $K_{\widehat{X}}$ is trivial, the holomorphic line bundle $\bigwedge^2 E$ is torsion. 
This implies that on a finite unramified cover of $\widehat{X}$ (still denoted by 
$\widehat{X}$ for notations convenience) we can (and will) assume that $\bigwedge^2 E$ is 
holomorphically trivial.

The above $\text{SL}_2({\mathbb C})$--structure on $\widehat{X}$
gives a holomorphic reduction $R'(\widehat{X})\, \subset\, R(\widehat{X})$ of the structure group of the
frame bundle $R(\widehat{X})$ of $\widehat{X}$ 
from $\text{GL}_n({\mathbb C})$ to $\text{SL}_2({\mathbb C})$; recall that the homomorphism 
$\text{SL}_2({\mathbb C})\,\longrightarrow\, \text{GL}_n({\mathbb C})$ is given by the 
$(n-1)$--th symmetric product of the standard representation, and this homomorphism is 
injective because $n$ is even. There is a  finite set of holomorphic tensors $\theta_1, \ldots, \theta_s$ on $\widehat{X}$ satisfying the 
condition that the $\text{SL}_2({\mathbb C})$--subbundle $R'(\widehat{X})\, \subset\,R(\widehat{X})$ consists of 
those frames that pointwise preserve each  $\theta_i$. Indeed, this is a consequence of Chevalley's 
theorem asserting that there exists a finite dimensional linear representation $W$ of 
$\text{GL}_n({\mathbb C})$, and an element $$\theta_0 \,\in\, W\, ,$$ such that the stabilizer 
of  the line ${\mathbb C} \theta_0$ is the image of the above homomorphism $\text{SL}_2({\mathbb 
C})\,\longrightarrow\, \text{GL}_n({\mathbb C})$ (see \cite[p.~80, Theorem 11.2]{Hu},
\cite[p.~40, Proposition 3.1(b)]{DMOS}; since
$\text{SL}_2({\mathbb C})$ does not have a nontrivial character, the line  ${\mathbb C} \theta_0$ 
is fixed pointwise.  The group  $\text{GL}_n({\mathbb C})$ being reductive,  we decompose $W$ as  a direct sum  $\bigoplus_{i=1}^s W_i$ of irreducible representations.  Now, since any irreducible 
representation $W_i$ of the reductive group $\text{GL}_n({\mathbb C})$ is a factor of a 
representation $({\mathbb C}^n)^{\otimes p_i} \otimes (({\mathbb C}^n)^*)^{\otimes q_i},$ for some 
integers $p_i,\,q_i\, \geq \, 0$ \cite[p.~40, Proposition 3.1(a)]{DMOS}, the above element $\theta_0$
gives rise to a finite set $\theta_1, \ldots ,  \theta_s$  of  holomorphic tensors
\begin{equation}\label{th}
\theta_i  \in H^0(\widehat{X},\,(T{\widehat{X}})^{\otimes p_i}\otimes (T^*{\widehat{X}})^{\otimes q_i})
\end{equation}
with $p_i,q_i\, \geq \, 0$. By construction, $\theta_1, \ldots, \theta_s$ are simultaneously  stabilized exactly by the frames lying in $R'(\widehat{X})$.

Consider the Levi--Civita connection on $\widehat{X}$ associated to the Ricci--flat K\"ahler 
metric $\varphi^*g$ in \eqref{e2}. It is known that the parallel transport for it preserves any 
holomorphic tensor on $\widehat{X}$ \cite[p.~50, Theorem 2.2.1]{LT}. In particular, $\theta_i$ in 
\eqref{th} are  parallel with respect to this connection. Hence we conclude that the subbundle 
$R'(\widehat{X})\, \subset\, R(\widehat{X})$ defining the $\text{SL}_2({\mathbb C})$-structure 
(considered as a $G$--structure) is invariant under the parallel transport by the Levi--Civita 
connection for $\varphi^*g$. This implies that the holonomy group of $\varphi^*g$ lies in the 
maximal compact subgroup of $\text{SL}_2({\mathbb C})$. Hence the holonomy group of $\varphi^*g$ 
lies in $\text{SU}(2)$.

From \eqref{e2} it follows that the holonomy of $\varphi^*g$ is
\begin{equation}\label{e3}
\text{Hol}(\varphi^*g)\,=\, \prod_{i=1}^p \text{Hol}(g_i)\, ,
\end{equation}
where $\text{Hol}(g_i)$ is the holonomy of $g_i$. As noted earlier,
\begin{itemize}
\item either $\text{Hol}(g_i)\,=\, \text{SU}(r_i)$, with $\dim_{\mathbb C} X_i \,=\, r_i\, \geq\, 2$, or

\item $\text{Hol}(g_i)\,=\, {\rm Sp}(\frac{r_i}{2})$, where $r_i\, =\, \dim_{\mathbb C} X_i$ is
even.
\end{itemize}
Therefore, the above observation, that $\text{Hol}(\varphi^*g)$ is contained in $\text{SU}(2)$,
and \eqref{e3} together imply that
\begin{enumerate}
\item either $(\widehat{X},\, \varphi^*g)\, =\, (T_l,\, g_0)$, or

\item $(\widehat{X},\, \varphi^*g)\, =\, (T_l,\, g_0) \times (X_1,\, g_1)$, where $X_1$ is a K3
surface equipped with a Ricci--flat K\"ahler metric $g_1$.
\end{enumerate}

If $(\widehat{X},\, \varphi^*g)\, =\, (T_l,\, g_0)$, then then proof of the theorem evidently is complete.

Therefore, we assume that
\begin{equation}\label{a1}
(\widehat{X},\, \varphi^*g)\, =\, (T_l,\, g_0) \times (X_1,\, g_1)\, ,
\end{equation}
where $X_1$ is a K3 surface equipped with a Ricci--flat K\"ahler metric $g_1$. Note that
$l\, \geq\, 2$ (because $l+2\,=\, n\, \geq\, 4$) and $l$ is even (because $n$ is so).

Since $\text{Hol}(g_1)\,=\, \text{SU}(2)$, we get  from (\ref{e3})  that $\text{Hol}(\varphi^*g)\,=\, \text{SU}(2)$. The holonomy of  $\varphi^*g$  is 
the image of the homomorphism
\begin{equation}\label{h0}
h_0\, :\, \text{SU}(2)\, \longrightarrow\, \text{SU}(n)
\end{equation}
given by the $(n-1)$--th symmetric power of the standard representation. The action of
$h_0(\text{SU}(2))$ on ${\mathbb C}^n$, obtained by restricting the standard action of $\text{SU}(n)$,
is irreducible. In particular, there are no nonzero
$\text{SU}(2)$--invariants in ${\mathbb C}^n$.  

On the other hand we have that:
\begin{itemize}
\item the direct summand of $T \widehat{X}$ given by the tangent bundle $TT_l$ is
preserved by the Levi--Civita connection on $T \widehat{X} $ corresponding to $\varphi^*g$, and

\item this direct summand of $T \widehat{X}$ given by $TT_l$ is generated by flat sections of $TT^l$.
\end{itemize}

Since $T \widehat{X}$ does not have any flat section, we conclude that $l=0$: a contradiction.
\end{proof} 

Recall that a  compact K\"ahler manifold of odd complex dimension bearing a holomorphic 
$\text{SL}_2({\mathbb C})$--structure also admits a holomorphic Riemannian metric and  inherits of  the associated 
holomorphic (Levi-Civita) affine connection. Those manifolds are known to have vanishing 
Chern classes  \cite{At} and, consequently, all of them  are covered by  compact complex tori 
\cite{IKO}.

Therefore  Theorem \ref{kahler even} has the following corollary.

\begin{corollary}\label{sl2}
Let $X$ be a compact K\"ahler manifold of complex dimension $n \,\geq\, 3$
bearing a holomorphic ${\rm SL}_2({\mathbb C})$--structure. Then $X$ admits a finite unramified
covering by a compact complex torus.
\end{corollary} 

We will prove an analog of Theorem \ref{kahler even} for Fujiki class $\mathcal C$ manifolds. 
Recall that a compact complex manifold $Y$ is in the class $\mathcal C$ of Fujiki if $Y$ is the 
image of a K\"ahler manifold through a holomorphic map. By a result of Varouchas in \cite{Va} we 
know that Fujiki class $\mathcal C$ manifolds are precisely those that are bimeromorphic to 
K\"ahler manifolds.

In order to use a result in \cite{BDG} (Theorem A) we will need to make a technical assumption. Let 
us recall some standard terminology:

\begin{definition} \label{technical}
Let $X$ be a compact complex manifold of dimension $n$, and let $$[\alpha]\,\in \,
H^{1,1}(X,\, {\mathbb R})$$ be a cohomology class represented by a
smooth, closed $(1,1)$-form $\alpha$.
\begin{enumerate}
\item $[\alpha]$ is \textit{numerically effective (nef)} if for any $\epsilon >0$, there exists a smooth representative
$\omega_{\epsilon}\,\in\, [\alpha]$ such that $\omega_{\epsilon}\,\ge\, -\epsilon \omega_X$, where $\omega_X$ is some
fixed (independent of $\epsilon$) hermitian metric on $X$.

\item $[\alpha]$ has \textit{positive self-intersection} if $\int_X \alpha^n\,>\,0$. 
\end{enumerate}
\end{definition}

It should be mentioned that Demailly and P\u{a}un conjectured the following (\cite[p.~1250, Conjecture~0.8]{DP}): If a complex compact 
manifold $X$ possesses a nef cohomology class $\lbrack \alpha\rbrack$ which has positive
self-intersection, then $X$ lies in the Fujiki class $\mathcal C$.

 The above  conjecture
would imply that a nef class $[\alpha]\,\in\, H^{1,1}(X,\,\mathbb R)$ on a compact complex manifold $X$  has positive self-intersection
if and only if $[\alpha]$ is big (i.e., 
there exists a closed $(1,1)$-current $T\,=\,\alpha+dd^c u\,\in \,[\alpha]$ such that $T\,\ge\, \omega_X$ in the sense of
currents, where $u\in L^1(X)$ and $\omega_X$  is some hermitian metric on $X$) \cite{DP} (see also \cite{BDG}, Corollary 2.6).  

It is not easy to decide which   manifolds  in Fujiki class $\mathcal C$  admit such a nef class with positive self-intersection.

The reader is referred to Section 2 of \cite{BDG} for a  detailed discussion
on Definition \ref{technical}.

\begin{theorem}\label{Fujiki}
Let $X$ be a compact complex manifold in Fujiki class $\mathcal C$ bearing a holomorphic 
${\rm GL}_2({\mathbb C})$--structure. Assume that there exists a cohomology class $\lbrack \alpha 
\rbrack\,\in\, H^{1,1}(X,\, \mathbb R)$ which is nef and has positive self-intersection.
Then the following two statements hold.
\begin{enumerate}
\item[(i)] If the complex dimension of $X$ is even and at least $4$, then there exists a non-empty 
Zariski open subset $\Omega\,\subset\, X$ admitting a flat K\"ahler metric.

\item[(ii)] If the complex dimension of $X$ is odd and the first Chern class of $X$ vanishes, then $X$ 
admits a finite unramified cover which is a torus.
\end{enumerate}
\end{theorem}

\begin{proof}
(i)~ Since the complex dimension of $X$ is even, $X$ inherits a twisted holomorphic 
symplectic form $ \omega\, \in\, H^0(X,\, \Omega^2_X \otimes L)$, defined by a a holomorphic 
line bundle $L$ and a non degenerate $L$-valued holomorphic two form $\omega$.

The proof of Theorem 2.5 in \cite{Is1} shows that $L$ inherits a holomorphic connection (see 
Remark 2.7 in \cite{Is1}) and hence its curvature, representing the first Chern class of 
$L$, is a holomorphic two form. In particular, $c_1(L)$ admits a representative which is a 
two form on $X$ of type $(2,0)$. On the other hand, starting with a Hermitian metric on $L$, 
the classical computation of the first Chern class using the associated Chern connection 
gives a representative of $c_1(L)$ which is a two form on $X$ of type $(1,1)$. Since on 
Fujiki class $\mathcal C$ manifolds (just as for K\"ahler manifolds) forms of different 
types are cohomologous only in the trivial class this implies that $c_1(L)=0$ and, 
consequently, as in the proof of Theorem \ref{kahler even}, $c_1(X)=0$.

Then Theorem A of \cite{BDG} constructs a closed, positive $(1,1)$-current on $X$,
lying in the class $[\alpha]$,
which induces a genuine Ricci-flat K\"ahler metric $g$ on a non-empty Zariski open subset 
$\Omega\,\subset\, X$. Furthermore, given any global holomorphic tensor $\theta\,\in\, 
H^0(X,\, (TX)^{\otimes p}\otimes T^*X)^{\otimes q})$, with $p,\, q\, \geq \, 0$, the restriction of 
$\theta$ to $\Omega$ is parallel with respect to the Levi--Civita connection on $\Omega$ associated
to $g$.

As in the proof of Theorem \ref{kahler even}, this implies that the  restricted holonomy group of Levi--Civita
connection for $g$ lies inside $\text{SU}(2)$.

Take any $u \,\in\, \Omega$. Using de Rham's local splitting theorem, there exists a local 
decomposition of an open neighborhood $U^u\, \subset\, \Omega$ of $u$ such that $(U^u,\, g)$
is a Riemannian product
\begin{equation}\label{e0p}
(U^u,\, g)\,=\, (U_0,\, g_0) \times \cdots \times (U_p,\, g_p)\, ,
\end{equation}
where $(U_0, \, g_0)$ is a flat K\"ahler manifold and $(U_i,\, g_i)$ is an irreducible K\"ahler 
manifold of complex dimension $r_i\, \geq\, 2$ for every $1\, \leq\, i\, \leq\, p$ (the reader is 
referred to \cite[Proposition 2.9]{GGK} for more details on this local K\"ahler decomposition). 
Since $g$ is Ricci flat, each $(U_i,\, g_i)$ is also Ricci flat. For every $1\, \leq\, i\, \leq\, 
p$, the restricted  holonomy $\text{Hol}_0(g_i)$ of the Levi--Civita connection for $g_i$ satisfies the 
following:
\begin{itemize}
\item either $\text{Hol}_0(g_i)\,=\, \text{SU}(r_i)$, or

\item $\text{Hol}_0(g_i)\,=\, {\rm Sp}(\frac{r_i}{2})$ (in this case 
$r_i$ is even);
\end{itemize}
see \cite[Proposition 5.3]{GGK}. The restricted  holonomy $\text{Hol}_0(g)$ of the Levi--Civita connection
for $g$ is
\begin{equation}\label{e1p}
\text{Hol}_0(g)\,=\, \prod_{i=1}^p \text{Hol}_0(g_i)\, .
\end{equation}

As observed above, $\text{Hol}_0(g)$ lies inside $\text{SU}(2)$. First assume that
$$\text{Hol}_0(g)\, \subsetneq\, \text{SU}(2)\, .$$
Therefore, $\dim \text{Hol}_0(g)\, < \,3$.
Now from \eqref{e1p} it follows that $\text{Hol}_0(g)\, =\, 1$ and there is no factor $(U_i,\, g_i)$ with $1\,\leq\, i\,\leq\, p$ in the decomposition \eqref{e0p} (because $\dim \text{Hol}_0(g_i)
\, \geq\, 3$ for every $1\,\leq\, i\,\leq\, p$). It follows
that $ (U^u,\, g)\,=\, (U_0,\, g_0)$ and hence $\Omega$ admits the flat K\"ahler metric $g$.

Therefore, now assume that 
$$\text{Hol}_0(g)\, =\, \text{SU}(2)\, .$$
From \eqref{e1p} we conclude the following:
\begin{itemize}
\item $p \,=\, 1$, meaning $(U^u,\, g)\,=\, (U_0,\, g_0) \times (U_1,\, g_1)$,

\item $\dim_{\mathbb C} U_1\,=\, 2$, and

\item $\text{Hol}_0(g_1)\,=\, \text{SU}(2)$.
\end{itemize}
Consider the homomorphism in \eqref{h0}. Consider ${\mathbb C}^n$ as a $\text{SU}(2)$--representation
using $h_0$ and the standard representation of $\text{SU}(n)$. This
$\text{SU}(2)$--representation is irreducible, in particular, there are no nonzero
$\text{SU}(2)$--invariants in ${\mathbb C}^n$. From this it can be deduced that $\dim U_0\,=\,
0$, where $U_0$ is the factor in \eqref{e0p}. Indeed,
\begin{itemize}
\item the direct summand of $TU^u$ given by the tangent bundle $TU_0$ is
preserved by the Levi--Civita connection on $TU^u$ corresponding to $g$, and

\item this direct summand of $TU^u$ given by $TU_0$ is generated by flat sections of $TU^u$.
\end{itemize}
Since $TU^u$ does not have any flat section, we conclude that $\dim_{\mathbb C}  U_0\,=\, 0$.

Therefore, $(U^u,\, g)\,=\, (U_1,\, g_1)$, and $\dim_{\mathbb C}  U^u \,=\, \dim_{\mathbb C}  U_1\,=\,2$. This contradicts
the assumption that $\dim_{\mathbb C} X\, \geq\, 4$.
Hence the proof of (i) is complete.

(ii)~ Fix a $\text{GL}_2({\mathbb C})$--structure on $X$. As the dimension of $X$ is odd, the 
$\text{GL}_2({\mathbb C})$--structure on $X$ produces a holomorphic conformal structure on $X$. 
Since the first Chern class of $X$ vanishes, by Theorem 1.5 of \cite{To}, there exists a finite 
unramified covering of $X$ with trivial canonical bundle. Replacing $X$ by this finite unramified 
covering we shall assume that $K_X$ is trivial.

As in the proof of Theorem \ref{kahler even}, this implies that, up to a finite unramified cover,  the $\text{GL}_2({\mathbb 
C})$--structure of $X$ is induced by a $\text{SL}_2({\mathbb C})$--structure on $X$. In particular, 
$X$ admits a holomorphic Riemannian metric. Now Theorem C of \cite{BDG} says that $X$ admits a 
finite unramified cover which is a compact complex torus.
\end{proof}

If $X$ is K\"ahler, the open subset $\Omega\,\subset\, X$ in Theorem \ref{Fujiki} (point 
(i)) is the entire manifold. 

We note that there are simply connected non-K\"ahler manifolds in Fujiki class $\mathcal C$
that admit a holomorphic symplectic form. Such examples were constructed in \cite[Example 21.7]{Huy}.

A compact complex manifold is called Moishezon if it is
bimeromorphic to a projective manifold \cite{Mo}. It is known that Moishezon manifolds
bearing a holomorphic Cartan geometry (in 
particular, a holomorphic conformal structure \cite{Sh}) are projective (see Corollary 
2 in \cite{BM}). Therefore a Moishezon manifold admitting a $\text{SL}_2({\mathbb C})$--structure 
is covered by an abelian variety.
 
We conjecture that the statement of Theorem \ref{Fujiki} holds for all Fujiki class $\mathcal 
C$ manifolds. In particular, we conjecture that a Fujiki class $\mathcal C$ manifold bearing a 
$\text{SL}_2({\mathbb C})$--structure is covered by a compact complex torus.

\section{$\text{GL}_2({\mathbb C})$-structures on K\"ahler-Einstein and Fano manifolds}\label{KE 
and Fano}

\begin{theorem}\label{KE}
Let $X$ be a compact K\"ahler--Einstein manifold, of complex dimension at least three,
bearing a holomorphic ${\rm GL}_2({\mathbb 
C})$-structure. Then one of the following three holds:
\begin{enumerate}
\item $X$ admits an unramified covering by a compact complex torus;

\item $X$ is the three dimensional quadric $Q_3$ equipped with its standard  ${\rm GL}_2({\mathbb 
C})$--structure;
 
\item $X$ is covered by the three-dimensional Lie ball $D_3$ (the noncompact dual of the Hermitian 
symmetric space $Q_3$) endowed with the standard ${\rm GL}_2({\mathbb C})$-structure.
\end{enumerate}
\end{theorem}

\begin{proof}
Let $n$ be the complex dimension of $X$.
Let $g$ be a K\"ahler--Einstein metric on $X$ with Einstein factor $e_g$. 

If $e_g\,=\,0$, then $X$ is Calabi--Yau. Up to a finite unramified covering we can assume that the 
canonical bundle $K_X$ is trivial \cite{Bo1,Be}, and hence $X$ admits a holomorphic 
$\text{SL}_2({\mathbb C})$--structure. Now Corollary \ref{sl2}
implies that $X$ is covered by a compact complex torus torus.

Assume now $e_g \neq 0$.

For any integers $m,\, l$, the Hermitian structure $H$ on $(TX)^{\otimes m}\otimes 
((TX)^*)^{\otimes l}$ induced by $g$ also satisfies the Hermitian--Einstein condition. The 
Einstein factor $e_H$ for the Hermitian--Einstein metric $H$
is $(m-l)e_g$. In particular, $e_H\,=\, 0$ when $m\,=\,l$.
Consequently, any holomorphic section $\psi$ of $(TX)^{\otimes m}\otimes 
((TX)^*)^{\otimes m}$ is flat with respect to the Chern connection on $(TX)^{\otimes m}\otimes 
((TX)^*)^{\otimes m}$ corresponding to the Hermitian--Einstein structure $H$ \cite[p. 50, 
Theorem 2.2.1]{LT}.

Fix a $\text{GL}_2({\mathbb C})$--structure on $X$. It produces a holomorphic reduction of 
the structure group of the frame bundle $R(X)$ defined by a $\text{GL}_2({\mathbb C})$ 
principal subbundle denoted by $R'(X)$; recall that the homomorphism $\text{GL}_2({\mathbb 
C}) \,\longrightarrow\, \text{GL}_n({\mathbb C})$ is given by the $(n-1)$--th symmetric 
product of the standard representation of $\text{GL}_2(\mathbb C)$. This holomorphic 
reduction $R'(X)$ induces a holomorphic reduction, to $\text{GL}_2(\mathbb C)$, of the 
structure group of $(TX)^*$, and hence we get a holomorphic reduction, to 
$\text{GL}_2(\mathbb C)$, of the structure group of $\text{End}(TX)\,=\, TX \otimes (TX)^*$.

Since the above holomorphic reduction of the structure group of $TX \otimes (TX)^*\,=\, \text{End}(TX)$ to
$\text{GL}_2(\mathbb C)$ is induced by a holomorphic reduction of the structure group of
$(TX)^*$ to $\text{GL}_2(\mathbb C)$, we conclude that this holomorphic reduction of the structure
group of $TX \otimes (TX)^*$ to $\text{GL}_2(\mathbb C)$ actually produces a reduction of
the structure group of $TX \otimes (TX)^*$ to the quotient group
$\text{PGL}_2(\mathbb C)$ of $\text{GL}_2(\mathbb C)$. Note that this is equivalent to the
statement that the action of $\text{GL}_2(\mathbb C)$ on $S^{n-1}({\mathbb C}^2)$ factors
through the quotient group $\text{PGL}_2(\mathbb C)$ of $\text{GL}_2(\mathbb C)$.

The group $\text{PGL}_2(\mathbb C)$ does not admit any nontrivial character.
Consequently, the above reduction of structure group of $TX \otimes (TX)^*$ to $\text{PGL}_2(\mathbb C)$
is given by a finite set of  holomorphic tensors $\psi_i$ (as in the 
proof of Theorem \ref{kahler even} this is a consequence of Chevalley's theorem and of the fact that the structure group $\text{PGL}_2(\mathbb C)$  does not admit any nontrivial character). As noted 
above, all the tensors $\psi_i$ are  parallel with respect to the Chern connection on $TX\otimes (TX)^*$ 
associated to the Hermitian--Einstein metric $H$. It should be clarified that we chose to work with
$\text{End}(TX)$ instead of $TX$, because had  we worked with $TX$, we would have obtained from Chevalley's theorem  a
holomorphic line bundle instead of the holomorphic tensors $\psi_i$. In that case, the above criterion for deciding
flatness won't be applicable.

Let $\text{ad}(R(X))$ and $\text{ad}(R'(X))$ be the adjoint vector bundles for the principal
bundles $R(X)$ and $R'(X)$ respectively.

The above observation that  all  tensors $\psi_i$ are  parallel, with respect to the Chern connection on $TX\otimes (TX)^*$
associated to the Hermitian--Einstein metric $H$, does not imply that the reduction $R'(X)$ is preserved
by the Hermitian--Einstein connection. But it does imply that $\text{ad}(R'(X))$ is
preserved by the connection on $\text{End}(TX)$ induced by the Hermitian--Einstein connection. We shall give
below an alternative direct argument for it.

Note that $\text{ad}(R(X))\,=\, \text{End}(TX)\,=\,
TX\otimes (TX)^*$, and the Lie algebra structure of the fibers of $\text{ad}(R(X))$ is the
Lie algebra structure of the fibers of $\text{End}(TX)$ given by the usual Lie bracket $(A,\, B)\, \longmapsto\, AB-BA$.
We have a holomorphic inclusion of Lie algebra bundles
$$
\text{ad}(R'(X))\, \hookrightarrow\, \text{ad}(R(X)) \,=\, TX\otimes (TX)^*
$$
induced by the above holomorphic reduction of structure group $R'(X)\, \subset\, R(X)$.
As noted before, the K\"ahler--Einstein structure $g$ produces a Hermitian--Einstein structure
$H$ on $ TX\otimes (TX)^*$. In particular, the vector bundle $\text{ad}(R(X))$ is polystable of degree zero.

On the other hand, we have $\text{degree}(\text{ad}(R'(X)))\,=\, 0$. Indeed, any
$\text{GL}_2(\mathbb C)$--invariant nondegenerate symmetric bilinear form on the Lie algebra
$M(2,{\mathbb C})$ of $\text{GL}_2(\mathbb C)$ (for example, $(A,\, B)\, \longmapsto\, \text{Tr}(AB)$) produces
a fiberwise nondegenerate symmetric bilinear form on
$\text{ad}(R'(X))$, which in turn holomorphically identifies $\text{ad}(R'(X))$ with $\text{ad}(R'(X))^*$. 

Since $\text{ad}(R'(X))$ is a subbundle of degree zero of the polystable vector bundle 
$TX\otimes (TX)^*$ of degree zero, the Hermitian--Einstein connection on $TX\otimes (TX)^*$ 
given by $H$ preserves this subbundle $\text{ad}(R'(X))$. From this it can be deduced that 
the Levi--Civita connection on $TX$ corresponding to the K\"ahler--Einstein metric $g$ 
induces a connection on the principal $\text{GL}_2({\mathbb C})$--bundle $R'(X)$. To see 
this we note the following general fact. Let
$A\, \subset\, B\, \subset\, C$ be Lie groups such that $A$ is normal in $C$. Let
$E_B$ be a principal $B$--bundle and $E_C\,=\, E_B(C)$ the principal $C$--bundle
obtained by extending the structure group of $E_B$. Since $A$ is normal in $C$, the
quotient $E_B/A$ (respectively, $E_C/A$) is a principal $B/A$--bundle (respectively,
$C/A$--bundle). Let $\widetilde\nabla$ be a connection on $E_C$ such that the connection
on the principal $C/A$--bundle $E_C/A$ induced by $\widetilde\nabla$ preserves
the subbundle $E_B/A$. Then $\widetilde\nabla$ preserves the subbundle $E_B\, \subset\, E_C$.
In our situation, $C\,=\, \text{GL}(n,{\mathbb C})$ and $A$ is its center, while $B$
is the image of $\text{GL}(2,{\mathbb C})$.

The above observation, that the Levi--Civita connection for $g$ induces a connection on  the principal 
$\text{GL}_2({\mathbb C})$--bundle $R'(X)$, implies that the holonomy group of $g$ lies in the subgroup
$$
\text{GL}_2({\mathbb C})\cap \text{U}(n)\, =\, \text{U}(2)\, \subset\, \text{U}(n)\, .
$$
Using de Rham local Riemannian decomposition (see Proposition 2.9 in \cite{GGK} for a proof adapted to the 
K\"ahler case), and Berger's list of  groups  (see Proposition 3.4.1 in \cite[p.~55]{Jo}) of irreducible 
holonomies, we conclude that this holonomy, namely a subgroup of $\text{U}(2)$, appears only for
locally symmetric Hermitian spaces. The classification of the holonomies of locally symmetric Hermitian
spaces shows that $X$ is
\begin{itemize}
\item either biholomorphic to the Hermitian symmetric space ${\rm SO}(5, \mathbb R)/({\rm SO}(2, \mathbb R)
\times{\rm SO}(3, \mathbb R))$ \cite[p.~312]{Bes} (this is the case \textbf{BD I} in the list) (recall that ${\rm SO}(5, \mathbb R)/({\rm SO}(2, \mathbb R) \times {\rm SO}(3, \mathbb R))$ is
biholomorphic to the quadric $Q_3$);

\item or $X$ is covered by the bounded domain which is the noncompact dual (recall that
the noncompact dual of $Q_3$ is $D_3\,=\, {\rm SO}_0(3,2)/
({\rm SO}(2, \mathbb R) \times{\rm SO}(3, \mathbb R))$.
\end{itemize}
The holonomy group of the K\"ahler-Einstein metric of the quadric $Q_3$ is ${\rm SO}(2, 
\mathbb R) \times{\rm SO}(3, \mathbb R)$ \cite[p.~312]{Bes} (case \textbf{BD I} in the list). Here ${\rm SO}(2, 
\mathbb R)$ acts on $\mathbb C^3$ by the one parameter group $\text{exp}(t J)$ with $J$ being the 
complex structure, while the action of ${\rm SO}(3, \mathbb R)$ on ${\mathbb C}^3$ is given by the 
complexification of the standard action of ${\rm SO}(3, \mathbb R)$ on ${\mathbb R}^3$. Consequently, 
the action of the covering ${\rm U}(2)$ of ${\rm SO}(2, \mathbb R) \times{\rm SO}(3, \mathbb R)$
on ${\mathbb C}^3$ coincides with the action on the second symmetric power $S^2({\mathbb C}^2)$ of
the standard representation. 

The $\text{GL}_2({\mathbb C})$--structure on the quadric $Q_3$ must be flat \cite{KO, Kl, HM,Ye}.
Since the quadric is simply connected this flat $\text{GL}_2({\mathbb C})$--structure coincides with the 
standard one \cite{Oc,HM,Ye}.

Also, the only holomorphic conformal structure on any compact manifold covered by the noncompact 
dual $D_3$ of the quadric is the standard one \cite{Kl,Ye}.
\end{proof} 

The next result deals with Fano manifolds. Recall that a Fano manifold is a compact complex projective 
manifold such that the anticanonical line bundle $K_X^{-1}$ is ample. Fano manifolds are known to be rationally 
connected \cite{Ca2}, \cite{KMM}, and they are simply connected \cite{Ca1}.

A basic invariant of a Fano manifold is its {\it index}, which is, by definition, the 
maximal positive integer $l$ such that the canonical line bundle $K_X$ is divisible by $l$ in the Picard group of $X$.

\begin{theorem} \label{Fano}
Let $X$ be a Fano manifold, of complex dimension $n \,\geq\, 3$, that admits a holomorphic 
$\text{GL}_2({\mathbb C})$--structure. Then $n\,=\,3$, and $X$ is biholomorphic to the quadric $Q_3$ (the 
${\rm GL}_2({\mathbb C})$--structure being the standard one).
\end{theorem} 

\begin{proof}
Let
$$TX \,\stackrel{\sim}{\longrightarrow}\, S^{n-1}(E)$$
be a holomorphic $\text{GL}_2({\mathbb C})$--structure on $X$, where $E$ is a holomorphic vector bundle
of rank two on $X$. A direct computation shows that
$$K_X\,= \, (\bigwedge\nolimits^2 E^*)^{\frac{n(n-1)}{2}}\, .$$
Hence the index of $X$ is at least $\frac{n(n-1)}{2}$.
It is a known fact that the index of a Fano manifold $Y$ of complex dimension $n$ is at most $n+1$. Moreover, the 
index is maximal ($=\, n+1$) if and only if $Y$ is biholomorphic to the projective space 
${\mathbb C}{\mathbb P}^n$, and the index equals $n$ if and only if $Y$ is biholomorphic to the quadric 
\cite{KO1}. These imply that
$n\,=\,3$ and $X$ is biholomorphic to the quadric $Q_3$. As in the proof of Theorem 
\ref{KE}, the results in \cite{HM,KO,Oc} imply that the only $\text{GL}_2({\mathbb C})$--structure on the quadric 
$Q_3$ is the standard one.
\end{proof} 

\section{Related open questions}

In this section we collect some open questions on compact complex manifolds bearing a holomorphic 
$\text{GL}_2({\mathbb C})$--structure or a holomorphic $\text{SL}_2({\mathbb C})$--structure.

\textbf{$\text{SL}_2({\mathbb C})$--structure on Fujiki class $\mathcal C$ manifolds.}

We think that the statement of Theorem \ref{Fujiki} holds without the technical assumption 
described in Definition \ref{technical} (the existence of a cohomology class $\lbrack \alpha 
\rbrack\,\in\, H^{1,1}(X,\, \mathbb R)$ which is nef and has positive self-intersection).

In particular, we conjecture that {\it an odd dimensional compact complex manifold $X$ in Fujiki 
class $\mathcal C$ bearing a holomorphic ${\rm SL}_2({\mathbb C})$--structure is covered by a 
compact torus}. Notice that in this case $X$ admits the Levi--Civita holomorphic affine connection 
corresponding to the associated holomorphic Riemannian metric, which implies, as in the K\"ahler 
case, that all Chern classes of $X$ vanish \cite{At}.

\textbf{$\text{SL}_2({\mathbb C})$--structure on compact complex manifolds.}
 
Recall that Ghys manifolds, constructed in \cite{Gh} by deformation of quotients of 
$\text{SL}_2({\mathbb C})$ by normal lattices, admit non-flat locally homogeneous holomorphic 
Riemannian metrics (or equivalently, $\text{SL}_2({\mathbb C})$--structures). It was proved in 
\cite{Du} that all holomorphic Riemannian metrics on compact complex threefolds are locally 
homogeneous.
 
We conjecture that {\it ${\rm SL}_2({\mathbb C})$--structures on compact complex manifolds of odd 
dimension are always locally homogeneous.}

\textbf{$\text{GL}_2({\mathbb C})$--structures on compact K\"ahler manifolds of odd dimension.}
 
Recall that these manifolds admit a holomorphic conformal structure. Flat conformal structures on 
compact projective manifolds were classified in \cite{JR2}: beside some projective surfaces, there are only 
the standard examples.
 
As for the conclusion in Theorem \ref{KE} and in Theorem \ref{Fano}, we think that {\it K\"ahler 
manifolds of odd complex dimension $\geq 5$ and bearing a holomorphic $\text{GL}_2({\mathbb 
C})$--structure are covered by compact tori.}

\section*{Acknowledgements}

The authors would like to thank Charles Boubel who kindly explained the geometric construction of the noncompact 
dual $D_3$ of the quadric $Q_3$ presented in Section \ref{contexte}.

The authors wish to thank the referee for very careful reading of the manuscript and helpful comments.

This work has been supported by the French government through the UCAJEDI Investments in the 
Future project managed by the National Research Agency (ANR) with the reference number 
ANR2152IDEX201. The first-named author is partially supported by a J. C. Bose Fellowship, and 
school of mathematics, TIFR, is supported by 12-R$\&$D-TFR-5.01-0500. The second-named author wishes to thank TIFR Mumbai, ICTS Bangalore and IISc Bangalore for hospitality.


\end{document}